\documentclass[12pt]{amsart}
\usepackage{amsmath,amssymb,amsthm,fullpage,xcolor,ulem,tikz}
\usetikzlibrary{calc}

\newcommand{\floor}[1]{\ensuremath{\left\lfloor #1\right\rfloor} }
\newcommand{\ceil}[1]{\ensuremath{\left\lceil #1\right\rceil} }
\newcommand{\size}[1]{\left \vert #1 \right \vert}
\newcommand{\comment}[1]{}

\newtheorem{lemma}{Lemma}[section]
\newtheorem{theorem}[lemma]{Theorem}

\newtheorem{corollary}[lemma]{Corollary}

\theoremstyle{definition}
\newtheorem{definition}[lemma]{Definition}

\DeclareMathOperator{\cinf}{c_{\infty}}

\def\cart{\, \Box \,}

\newcommand{\cnum}{c}

\begin{document}
\title{Catching an infinitely fast robber on a grid}

\author{William B. Kinnersley}
\address{Department of Mathematics and Applied Mathematical Sciences, University of Rhode Island, University of Rhode Island, Kingston, RI, USA, 02881}
\email{\tt billk@uri.edu}

\author{Nikolas Townsend}
\address{Department of Mathematics and Applied Mathematical Sciences, University of Rhode Island, University of Rhode Island, Kingston, RI, USA, 02881}
\email{\tt townsendn@uri.edu}


\begin{abstract}
We consider a variant of Cops and Robbers in which the robber may traverse as many edges as he likes in each turn, with the constraint that he cannot pass through any vertex occupied by a cop.  We study this model on several classes of grid-like graphs.  In particular, we determine the cop numbers for two-dimensional Cartesian grids and tori up to an additive constant, and we give asymptotic bounds for the cop numbers of higher-dimensional grids and hypercubes. 
\end{abstract}

\maketitle

\begin{section}{Introduction}

{\it Cops and Robbers} is a well-studied model of pursuit and evasion in which a team of one or more cops attempts to capture a single robber on a graph $G$.  The cops and robber occupy vertices of $G$; multiple cops may occupy the same vertex, and both players know the positions of the cops and robber.  For convenience, throughout the paper we use the pronouns he/him/his to refer to the robber and she/her/hers to refer to the cops. The cops begin the game by choosing their initial positions, after which the robber chooses his initial position in response.  Thenceforth the game proceeds in {\it rounds}, each consisting of a {\it cop turn} followed by a {\it robber turn}.  On the cops' turn, each cop may either remain on her current vertex or move to a neighboring vertex; likewise, on the robber's turn, he may remain on his current vertex or move to a neighboring vertex.  If any cop ever occupies the same vertex as the robber, we say that the cop {\it captures} the robber, and the cops win.  Conversely, the robber wins if he can perpetually evade capture.

Given a graph $G$, it is natural to ask how many cops are needed to guarantee that the cops win the game on $G$.  The {\it cop number} of $G$, denoted $\cnum(G)$, is the minimum number of cops needed to guarantee capture of a robber on $G$ no matter how the robber plays.  Quilliot~\cite{Qui78}, as well as Nowakowski and Winkler~\cite{NW83}, independently introduced the game of Cops and Robbers and characterized graphs with cop number 1.  Aigner and Fromme~\cite{AF84} first introduced the notion of the cop number; since then, the cop number has been extensively studied. For more background on the game, we direct the reader to~\cite{BN11}.

Cops and Robbers has spawned many variants, each of which models pursuit and evasion in a slightly different context.  One natural variant permits the robber to move faster than the cops: a robber with \textit{speed $s$} (for some positive integer $s$) may, on each robber turn, traverse any cop-free path of length at most $s$; in other words, the robber may take up to $s$ steps on his turn, provided that he never passes through a vertex occupied by a cop.  (Note that the case $s=1$ corresponds to the original model of Cops and Robbers.)  A robber with \textit{speed $\infty$} may follow any cop-free path he likes, regardless of its length.  As in the usual model of Cops and Robbers, it is natural to ask how many cops are needed to capture a robber with speed $s$; this quantity is the {\it speed-$s$ cop number} of $G$, denoted $c_s(G)$.  When $s=\infty$, we refer to $c_{\infty}(G)$ as the {\it infinite-speed cop number}.



The speed-$s$ cop number was formally introduced by Fomin, Golovach, Kratovchv\'{i}l, Nisse, and Suchan in \cite{FGKNS10}.  They investigated the computational complexity of computing $c_s$ when $s \ge 2$; additionally, they proved that $c_2(P_n \cart P_n) = \Omega(\sqrt{\log n})$, thereby showing that $c_2$ is unbounded on the class of planar graphs (in contrast to the usual cop number, which is bounded above by 3 \cite{AF84}).  The latter result was extended by Balister, Bollob\'{a}s, Narayanan, and Shaw \cite{BBNS17}, who showed that whenever $s$ exceeds some absolute constant, there exists some $\eta$ such that $c_s(P_n \cart P_n) \ge \textrm{exp}\left(\frac{\eta \log n}{\log \log n}\right)$ when $n$ is sufficiently large.  Frieze, Krivilevich, and Loh \cite{FKL12} constructed $n$-vertex graphs with speed-$s$ cop number in $\Omega(n^{(s-3)/(s-2)})$; this result was subsequently improved by Alon and Mehrabian \cite{AM11}, who constructed graphs with speed-$s$ cop number in $\Omega(n^{s/(s+1)})$.  The infinite-speed cop number, $\cinf$, has received a particularly great deal of attention.  It has been studied on many classes of graphs, including planar graphs \cite{AM15}, chordal graphs \cite{Meh15}, interval graphs \cite{Meh15}, expander graphs \cite{Meh12}, random graphs \cite{AM15, Meh12}, and Cartesian products of paths and cycles \cite{Meh12}.  Those graphs for which $\cinf(G) = 1$ were characterized in \cite{Meh12}, and a construction of graphs satisfying $\cinf(G) = \Theta(V(G))$ was given in \cite{FKL12}.  Results on the computational complexity of determining $\cinf(G)$ appear in \cite{DGK19} and \cite{FGKNS10}.

In this paper, we analyze the infinite-speed game on several families of Cartesian products of paths and cycles, extending earlier work of Mehrabian in \cite{Meh12}.  Mehrabian showed that when $G$ is the $k$-fold Cartesian product of $n$-vertex paths -- that is, the $k$-dimensional Cartesian grid -- we have $\frac{1}{4k^2}n^{k-1} \le \cinf(G) \le n^{k-1}$; when $G$ is the $k$-fold Cartesian product of $n$-vertex cycles, he showed $\frac{1}{2k^2}n^{k-1} \le \cinf(G) \le 2n^{k-1}$.  Finally, when $G$ is the $k$-dimensional hypercube -- that is, the $k$-fold product of $P_2$ -- he proved that for some constants $\eta_1$ and $\eta_2$, we have $\frac{\eta_1}{k^{3/2}}2^k \le \cinf(G) \le \frac{\eta_2}{k}2^k$.  We aim to tighten these bounds.

In Section \ref{sec:2D_grids}, we consider two-dimensional Cartesian grids.  We show in Theorem \ref{thm:2D_grids} that $\cinf(P_n \cart P_n)$ is $n-1$ when $n$ is odd, and either $n-1$ or $n$ when $n$ is even.  Similar arguments show that the infinite-speed cop number for the two-dimensional Cartesian torus, $C_n \cart C_n$, is between $2n-24$ and $2n$ (provided $n \ge 18$); see Theorem \ref{thm:torus}.  Section \ref{sec:3D_grids} deals with higher-dimensional grids, on which the game is substantially more complex.  Our main result in this section is that for sufficiently large $n$, we have $0.7172n^2 \le c_{\infty}(P_n \cart P_n \cart P_n) \le (0.75+o(1))n^2$.  (Refer to Theorem \ref{thm:3D_grid_upper} for the upper bound and Theorem \ref{thm:3D_grid_lower} for the lower bound.)  Our upper bound generalizes to grids of higher dimension: see Theorem \ref{thm:4D_grid_upper}.  As a consequence of Theorem \ref{thm:3D_grid_lower}, we also obtain a lower bound on the treewidth of $P_n \cart P_n \cart P_n$; see Corollary \ref{cor:treewidth}.  In Section \ref{sec:hypercubes}, we consider hypercubes.  We use a potential function argument to improve upon the lower bound from \cite{Meh12}: our Theorem \ref{thm:hypercubes} states that when $G$ is the $k$-dimensional hypercube, $c_{\infty}(G) \ge \frac{2^{k-3}}{k\ln k}$.  Finally, in Section \ref{sec:conclusion}, we propose several open problems and directions for future study.\\


\end{section}

\begin{section}{Two-Dimensional Grids}\label{sec:2D_grids}

We begin the paper with a discussion of the infinite-speed robber game played on two-dimensional Cartesian grids, i.e. graphs of the form $P_m \cart P_n$.  Maamoun and Meyniel \cite{MM87} showed that in the speed-1 game, two cops suffice to capture a robber on any two-dimensional grid.  One might expect that as the speed of the robber increases, so too does the number of cops needed to capture him.  This is indeed the case: Fomin et al. \cite{FGKNS10} showed that $c_2(P_n \cart P_n) = \Omega(\sqrt{\log n})$.  In this section, we determine the infinite-speed cop number of $P_n \cart P_n$ up to an additive constant; in particular, we show that $c_{\infty}(P_n \cart P_n) \in \{n-1,n\}$.

We view the vertex set of $P_n \cart P_n$ as the set $\{(x,y) : 1 \le x\le n, \, 1\le y \le n\}$, and we imagine that the vertices are laid out in an $n \times n$ rectangular grid, with $(x,y)$ denoting the $x$th vertex from the left, in the $y$th row from the top.  Thus motivated, we define the {\it $i$th row} of $P_n \cart P_n$ to be the path consisting of all vertices of the form $(x,i)$ for $1 \le x \le n$; similarly, the {\it $j$th column} is the path consisting of vertices $(j,y)$ for $1 \le y \le n$.  The {\it left half} of $P_n \cart P_n$ consists of the first $\ceil{n/2}$ columns, while the {\it right half} consists of the last $\floor{n/2}$.  We say that vertex $(x,y)$ is {\it above} (resp. below, left of, right of) $(x',y')$ if $y < y'$ (resp. $y > y'$, $x < x'$, $x > x'$).

Mehrabian \cite{Meh12} observed that $\cinf(P_n \cart P_n) \le n$, since $n$ cops can begin by occupying every vertex in the top row, then move to occupy every vertex in the second row, then every vertex in the third row, and so on; the robber is forced to remain below the cops and must eventually be captured.
\begin{theorem}[\cite{Meh12}, Theorem 5.1(b)]\label{thm:2D_upper_Mehrabian}
For all positive integers $n$, we have $\cinf(P_n \cart P_n) \le n$.\looseness=-1
\end{theorem}

When $n$ is odd, this upper bound can be slightly improved.

\begin{theorem}\label{thm:2D_upper_odd}
If $n$ is odd and at least 3, then $\cinf(P_n \cart P_n) \le n-1$.
\end{theorem}
\begin{proof}
We give a strategy for $n-1$ cops to capture a robber on $G$.  At the beginning of the game, for each $k$ in $\{1, \dots, (n-1)/2\}$, two cops begin on vertex $(2k,2k)$, as depicted in Figure \ref{fig:2D_grid_upper}(a).  Suppose the robber begins on vertex $(x,y)$.  We consider three cases.
\begin{itemize}
\item {\it Case 1: $x<y$}.  For each $k$ in $\{1, \dots, (n-1)/2\}$, the two cops on vertex $(2k,2k)$ move to vertices $(2k-1,2k)$ and $(2k,2k+1)$.  Cops now occupy vertices $(1,2), (2,3), \dots, (n-1,n)$.  Let $S$ denote the set of cop positions, and note that $S$ is a cut-set of $G$: the component of $G\setminus S$ above the cops contains all vertices $(i,j)$ with $i \ge j$, while the component below contains all vertices $(i,j)$ with $i < j-1$.  Since $x<y$, either some cop has moved onto the robber's vertex (and the cops have won), or the robber is in the component below the cops.  On the robber's turn, he cannot pass through any vertex occupied by a cop, so he must remain in the component below the cops.  On the following cop turn, if the robber has not yet been captured, then the cops in columns $1, 2, \dots, n-2$ each move down one step.  On the cops' next turn -- supposing again that the robber has not yet been captured -- the cops in columns $1, 2, \dots, n-3$ each move down one step.  The cops continue in this manner until the robber is captured.  (Refer to Figure \ref{fig:2D_grid_upper}(b).)
\medskip
\item {\it Case 2: $x>y$}.  For this case, the cops can employ a strategy symmetric to that used in Case 1.
\medskip
\item {\it Case 3: $x=y$}.  If $x$ and $y$ are even, then the robber occupies the same vertex as some cop and the game is already over, so suppose instead that $x=y=2k-1$ for some $k$ in $\{1, \dots, (n+1)/2\}$.  If $k=1$, then the two cops on vertex $(2,2)$ move to vertices $(1,2)$ and $(2,1)$.  The robber is forced to remain on vertex $(1,1)$, so the cops can capture him on their ensuing turn.  Likewise, if $k=(n+1)/2$, then the cops on $(n-1,n-1)$ move to $(n-1,n)$ and $(n,n-1)$ and win on their next turn.  Finally, suppose $2 \le k \le (n-1)/2$.  The two cops on $(2k-2,2k-2)$ move to $(2k-1,2k-2)$ and $(2k-2,2k-1)$, while the two cops on $(2k,2k)$ move to $(2k,2k-1)$ and $(2k-1,2k)$.  The cops now occupy all four neighbors of the the robber's vertex, so once again he must remain on that vertex and the cops can capture him on their next turn. (Refer to Figure \ref{fig:2D_grid_upper}(c).)
\end{itemize}
\end{proof}

\begin{center}
\begin{figure}[ht]
\includegraphics{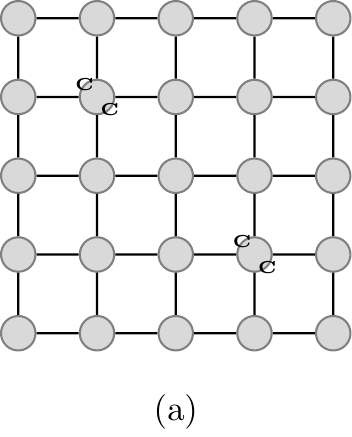}
\quad\quad\quad\quad\quad
\includegraphics{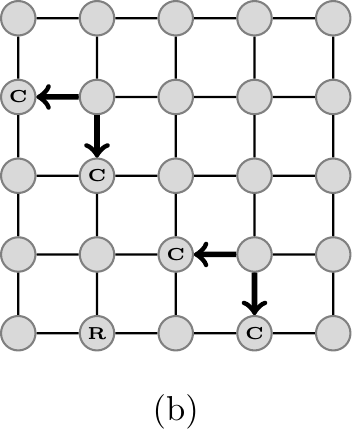}
\quad\quad\quad\quad\quad
\includegraphics{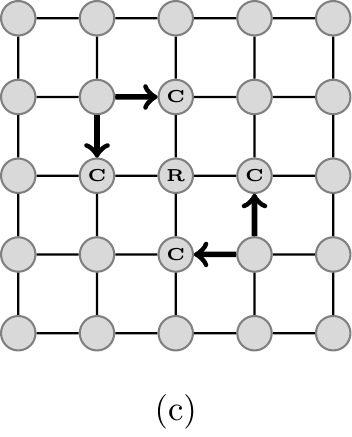}
\caption{Cop strategy in Theorem \ref{thm:2D_upper_odd} (for $n=5$).  Figure (a) depicts the cops' initial position; (b) depicts the cops' cop response to a robber below the main diagonal; (c) depicts the cops' response to a robber on the main diagonal.}
\label{fig:2D_grid_upper}
\end{figure}
\end{center}

Having established an upper bound on $\cinf(P_n \cart P_n)$, we seek to prove a matching lower bound.  As it turns out, this is substantially more difficult: proving a lower bound on $\cinf$ requires giving a strategy for the robber to evade a large number of cops, and the freedom of movement enjoyed by an infinite-speed robber enables fairly intricate strategies.

In general, the robber will want to seek out portions of the graph that contain no cops.  Given a graph $G$, we call a subgraph $H$ of $G$ {\it empty} at some point during the game if, at that time, no cops occupy vertices of $H$.  We say that the robber \textit{can reach} or \textit{has access to} a vertex $v$ if there is a cop-free path joining the robber's current position to $v$; the robber can reach a subgraph $H$ of $G$ if he can reach any vertex of $H$. 

\begin{theorem}\label{thm:2D_lower}
For all positive integers $n$, we have $\cinf(P_n \cart P_n) \ge n-1$.
\end{theorem}
\begin{proof}
The claim is clear by inspection when $n \le 3$, so suppose $n\ge 4$.  Let $G = P_n \cart P_n$; we give a strategy for the robber to evade $n-2$ cops on $G$.  Let $v_0$ denote the robber's initial position, and for $i \ge 1$, let $v_i$ denote the robber's position after his $i$th move.  The robber will choose each $v_i$ so that the following properties are satisfied:
\begin{enumerate}
\item[(1)] If the robber can reach an empty row of $G$ on his $i$th turn, then he can reach $v_i$.
\item[(2)] No cop is adjacent to $v_i$.
\item[(3)] After the cops' $(i+1)$st turn, the robber can reach an empty row of $G$ from $v_i$.
\end{enumerate}
If the robber can choose the $v_i$ in this way, then he can perpetually evade the cops: property (2) implies that the robber cannot be captured on the cops' $(i+1)$st turn, and properties (1) and (3) together imply that the robber can always reach $v_{i+1}$ from $v_i$, i.e. that he can legally move to $v_{i+1}$ with his $(i+1)$st turn.

Fix $i \ge 0$, and consider the state of the game after the cops' $i$th turn.  We explain how the robber chooses $v_i$.  We consider two cases. Case 1 covers the situation where there is a group of consecutive rows containing relatively few cops; in this case, the robber will be able to choose $v_i$ such that there are too few cops ``nearby'' to prevent him from reaching an empty row of $G$ after the cops' turn.  Case 2 deals with the situation where no cop occupies any vertices in the top row, bottom row, leftmost column, or rightmost column; here the robber will be able to find a suitable $v_i$ on the outer boundary of the grid.  Together, these cases cover all possible cop configurations.

{\bf Case 1:} For some $k$ in $\{2, 3, \dots, n-2\}$, either the top $k$ rows of $G$ contain no more than $k-2$ cops between them, or the bottom $k$ rows contain no more than $k-2$ cops, or both.  
Without loss of generality, suppose the top $k$ rows contain at most $k-2$ cops.  Within these rows, either the left half or the right half (or both) must contain at most $\floor{\frac{k-2}{2}}$ cops.  Suppose this is true of the right half; a similar argument suffices if it is true for the left half.  We refer to the top $k$ rows of the right half of $G$ as the robber's {\it sector} of $G$ and denote this subgraph by $S$.

If the top two rows of $S$ contain no cops, then the robber chooses $v_i$ to be any vertex in the top row of $S$ that doesn't belong to the leftmost column of $S$.  This choice of $v_i$ clearly satisfies property (2).  It satisfies property (1) as well: if the robber can reach an empty row of $G$, then he can use this empty row to reach an empty column (which must exist since there are $n$ columns but only $n-2$ cops), and he can use this empty column to reach the top row of $G$ and subsequently $v_i$.  Finally, to see that it satisfies property (3), note that since the top two rows of $G$ are empty before the cops' $(i+1)$st turn, the top row remains empty after the cops' $(i+1)$st turn.

Suppose instead that the top two rows of $S$ contain at least one cop.  $S$ contains $k$ rows and at most $\floor{\frac{k-2}{2}}$ cops. Note that $S$ must have at least 4 rows: it contains no more than $\floor{\frac{k-2}{2}}$ cops, but the top two rows contain at least one cop, hence $\floor{\frac{k-2}{2}} \ge 1$ and so $k \ge 4$.  We claim that there must be some four consecutive rows within $S$ that contain at most one cop between them.  To prove this, let $k = 4d+r$ for some positive integer $d$ and some $r \in \{0, \dots, 3\}$; we consider two cases.
\begin{itemize}
\item Suppose $r \in \{0, 1\}$.  In this case, $\floor{\frac{k-2}{2}} = 2d-1$.  Partition the top $4d$ rows of $S$ into $d$ groups of four rows apiece; since there are at most $2d-1$ cops, at least one of these groups of four rows must have fewer than two cops.
\medskip
\item Suppose $r \in \{2, 3\}$.  By assumption, the top two rows of $S$ contain at least one cop.  Hence, the remaining $k-2$ rows contain at least $\floor{\frac{k-2}{2}}-1$ cops, which simplifies to $2d-1$.  As in the previous case, we may partition the next $4d$ rows of $S$ into groups of four, one of which must have fewer than two cops.
\end{itemize}
\medskip
The robber now chooses $v_i$ to be any vertex within the second and third of these four rows that is not currently adjacent to a cop and does not belong to the leftmost column of $S$.  (To see that this is possible, first note that since $S$ has at least four rows, we must have $n \ge 6$, hence $S$ has at least three columns.  Even after discarding the leftmost column of $S$, at least two columns remain; the robber has at least four potential destinations to choose from, and a single cop cannot be adjacent to all of them.) 

The robber's choice of $v_i$ clearly satisfies property (2).  To see that it satisfies property (1), suppose the robber can reach an empty row of $G$.  Since $G$ has $n$ columns and at most $n-2$ cops, it must have an empty column, which the robber can reach using the empty row.  There must also be an empty row among the top $k$ rows of $G$; the robber can reach such a row using the empty column of $G$.  Next, the robber can use this empty row to reach an empty column of $S$ (which must exist since $S$ contains at least $\floor{n/2}$ columns and at most $\floor{\frac{k-2}{2}}$ cops).  Finally, since $v_i$ is either in or adjacent to an empty row of $S$, the robber can use the empty columns of $S$ to reach this row and subsequently $v_i$.

It remains to be shown that the robber's choice of $v_i$ satisfies property (3).  Let $S'$ denote $S$ with the bottom row and leftmost column removed.  Consider the state of the game after the cops' $(i+1)$st turn.  Although some cops may have entered the bottom row or leftmost column of $S$ on their turn, no cops outside $S$ can have entered $S'$.  Consequently, $S'$ has at most $\floor{\frac{k-2}{2}}$ cops.  Before the cops' move, $v_i$ belonged to either the second or third of four consecutive rows in $S'$ with at most one cop between them; hence, after the cops' move, the second and third rows still have at most one cop between them.  Consequently, either the row of $S'$ that contains $v_i$ is empty, or the row directly above or below $v_i$ is empty.  Either way, the robber can clearly reach an empty row of $S'$.  Note that $S'$ contains 
at least $\floor{n/2}-1$ columns and, as noted above, has at most $\floor{\frac{k-2}{2}}$ cops.  Since
$$\floor{\frac{k-2}{2}} = \floor{\frac{k}{2}} - 1 \le \floor{\frac{n-2}{2}} - 1 = \floor{\frac{n}{2}} - 2 < \floor{\frac{n}{2}} - 1,$$
$S'$ has at least one empty column.  Since the robber can reach an empty row of $S'$, he can thus also reach this empty column of $S'$.  Now note that the top $k$ rows of $G$ contained at most $k-2$ cops before the cops' move; some cops may have entered the $k$th row of $G$, but the top $k-1$ rows must still have at most $k-2$ cops.  Consequently, one of the top $k-1$ rows of $G$ is empty.  Since the robber can reach an empty column of $S'$, he can use it to reach this empty row of $G$.

Thus, the robber's choice of $v_i$ satisfies properties (1), (2), and (3), as claimed.

\medskip
{\bf Case 2:} There is no $k$ in $\{2, 3, \dots, n-2\}$ such that either the top $k$ or the bottom $k$ rows of $G$ contain at most $k-2$ cops.  In this case, the top two rows of $G$ cannot both be empty, since then we could take $k=2$ and use the top two rows of $G$.  The top two rows of $G$ cannot contain two or more cops between them, since then we could take $k=n-2$ and use the bottom $n-2$ rows of $G$.  Thus, the top two rows of $G$ contain exactly one cop between them.  Similarly, the bottom two rows of $G$ must also contain exactly one cop between them.  Every other row of $G$ must contain exactly one cop: if row $j$ were to contain more than one cop, for some $j$ in $\{3, \dots, n-2\}$, then either the top $j-1$ rows would contain at most $j-3$ cops, or the bottom $n-j$ rows would contain at most $n-j-2$ cops.  By symmetry, we may also assume that there is no $k$ in $\{2, 3, \dots, n-2\}$ such that either the leftmost $k$ or the rightmost $k$ columns of $G$ contain at most $k-2$ cops; otherwise, the robber could choose $v_i$ as in Case 1, except with the roles of rows and columns switched.  (Note that properties (1)-(3) are symmetric with respect to rows and columns, because the robber can reach an empty row if and only if he can reach an empty column.)  Hence, we may also conclude that the leftmost two columns contain exactly one cop, the rightmost two columns contain exactly one cop, and every other column contains exactly one cop.

The robber chooses $v_i$ to be any vertex that belongs to whichever of the top two rows is empty, is not adjacent to a cop, and is not in the leftmost or rightmost column of $G$.  It is clear that this choice of $v_i$ satisfies property (2). It also satisfies property (1): if the robber can reach an empty row of $G$, then he can reach an empty column, use this column to reach whichever of the top two rows is empty, and proceed from there to $v_i$.  

Finally, consider property (3).  After the cops' move, if one of the top two rows is empty, then the robber can clearly reach an empty row and property (3) is satisfied.  Otherwise, it must be that the cop in the third row moved up into the second row.  Say this cop occupies the $k$th column.  We claim that the robber can either reach all columns to the left of the $k$th column or all columns to the right.  Since neither of the top two rows is empty, each must contain exactly one cop.  If the robber is strictly to the left of the $k$th column, then he can move into the second row (if he isn't already there) and use it to reach all columns to the left of the $k$th column.  Similarly, if he is to the right of the $k$th column, then he can reach all columns to the right.  Finally, if the robber is in the $k$th column, then he must be in the top row; if the cop in the top row is to the left of the robber then the robber can reach all columns to the right of the $k$th column, and otherwise he can reach all columns to the left.

Without loss of generality, suppose the robber can reach all columns to the left of the $k$th column.  Before the cops' move, one of the leftmost two columns was empty; in particular, there was an empty column to the left of the $k$th column.  We claim that there is still an empty column to the left of the $k$th column; it would then follow that the robber can reach an empty column, from which he could reach an empty row, so property (3) would be satisfied.  

If the leftmost column is empty then we are done, so suppose otherwise.  Choose the least $\ell$ greater than 1 such that the cop in the $\ell$th column did not move to the left; note that $\ell \le k$, since the cop in the $k$th column moved up, not left.  In order for neither of the top two columns to be empty, it must be that the cop in the third column moved left into the second column.  Now, in order for the 3rd column not to be empty, the cop in the fourth column must have moved left into the third column, and so on.  It follows that the $(\ell-1)$th column must be empty, since the cop in the $(\ell-1)$th column moved left, but the cop in the $\ell$th column did not.  Since $\ell-1 < k$, we have shown that there is an empty column to the left of the $k$th column, as claimed.  It follows that the robber's choice of $v_i$ satisfies properties (1), (2), and (3).
\end{proof}

The following theorem summarizes our bounds on $\cinf(P_n \cart P_n)$.

\begin{theorem}\label{thm:2D_grids}
If $n \le 2$, then $\cinf(P_n \cart P_n) = n$; if $n$ is odd and at least 3, then $\cinf(P_n \cart P_n) = n-1$; if $n$ is even and at least 4, then $\cinf(P_n \cart P_n) \in \{n-1, n\}$.
\end{theorem}
\begin{proof}
Let $G = P_n \cart P_n$.  If $n\le 2$, then by inspection $\cinf(G) = n$.  If $n$ is odd and at least 3, then Theorems \ref{thm:2D_upper_odd} and \ref{thm:2D_lower} together give $n-1 \le \cinf(G) \le n-1$; if $n$ is even and at least 4, then Theorem \ref{thm:2D_upper_Mehrabian} and \ref{thm:2D_lower} together give $n-1 \le \cinf(G) \le n$.
\end{proof}

Although Theorem \ref{thm:2D_grids} does not establish the exact value of $\cinf(P_n \cart P_n)$ when $n$ is even, we suspect that in fact $\cinf(P_n\cart P_n)=n$ for all even $n$.  However, we were unable to find a simple proof of this, even for $n \in \{4, 6, 8\}$. 

Although we have thus far focused exclusively on $n \times n$ grids, the game behaves quite similarly on general $m \times n$ grids.  To formalize this, we first establish that $\cinf$ is monotone with respect to retraction; the proof of this fact is quite similar to an argument used by Nowakowski and Winkler \cite{NW83} in their characterization of cop-win graphs.  Given a graph $G$ and subgraph $H$, a \textit{retraction} from $G$ to $H$ is a homomorphism $\phi: G \rightarrow H$ such that $\phi(v) = v$ for all $v \in V(H)$.  When there exists a retraction from $G$ to $H$, we say that $H$ is a \textit{retract} of $G$. 

\begin{theorem}\label{thm:monotone}
If $H$ is a retract of $G$, then $\cinf(G)\ge \cinf(H)$.
\end{theorem}
\begin{proof}
Let $k = \cinf(H)$; we will show that the robber has a winning strategy on $G$ against $k-1$ cops, from which it immediately follows that $\cinf(G) \ge k = \cinf(H)$.  

Let $\phi$ be a retraction from $G$ to $H$.  Throughout the game, the robber will remain in $H$ and will use a winning strategy on $H$ to evade the cops on $G$.  On each robber turn, if the cops occupy vertices $v_1, \dots, v_k$ in $G$, then the robber imagines that they in fact occupy vertices $\phi(v_1), \dots, \phi(v_k)$ in $H$, and he moves according to a winning strategy in $H$.  When a cop moves from $v$ to $v'$ in $G$, the robber imagines that they move from $\phi(v)$ to $\phi(v')$ in $H$.  Since $v$ must be adjacent to $v'$ in $G$ and since $\phi$ is a homomorphism, $\phi(v)$ must be adjacent to $\phi(v')$ in $H$; consequently, every legal cop move in $G$ corresponds to a legal move of the cops' images on $H$.  Since the robber follows a winning strategy on $G$, he indefinitely evades the cops' images; since each vertex of $H$ is its own image under $\phi$, it follows that the robber indefinitely evades the cops themselves, meaning that he wins the game on $G$.
\end{proof}

The cop strategy used to show that $\cinf(P_n \cart P_n) \le n$ extends naturally to $P_m \cart P_n$ and shows that $\cinf(P_m \cart P_n) \le \min\{m,n\}$.  Additionally, for any $k\geq2$ and $n_1,n_2,\dots,n_k\geq n\geq 3$, there is a natural retraction $\phi:P_{n_1}\cart P_{n_2}\cart \dots\cart P_{n_k}\to P_n\cart P_n\cart \dots P_n$ given by $\phi((x_1,x_2,\dots,x_k))=(\min\{x_1,n\},\min\{x_2,n\},\dots,\min\{x_k,n\})$.  Thus, Theorems \ref{thm:2D_grids} and \ref{thm:monotone} together show that $\cinf(P_m \cart P_n) \ge \min\{m,n\}-1$.   

\begin{corollary}\label{cor:general_grids}
For all positive integers $m$ and $n$, 
$$\min\{m,n\}-1 \le \cinf(P_m \cart P_n) \le \min\{m,n\}.$$
\end{corollary}

We conclude the section with a brief look at the discrete torus $C_n \cart C_n$.  We can establish lower and upper bounds on $\cinf(C_n \cart C_n)$ with arguments very similar to those used for $P_n \cart P_n$.  As with $P_n \cart P_n$, we view the vertex set of $C_n \cart C_n$ as the set $\{(x,y) : 1 \le x\le n, \, 1 \le y \le n\}$.  The {\it $i$th row} is the cycle consisting of all vertices of the form $(x,i)$ for $1 \le x \le n$, and the {\it $j$th column} is the cycle consisting of $(j,y)$ for $1 \le y \le n$.  

\begin{theorem}\label{thm:torus}
    For all $n\geq 18$, we have $2n-24 \le \cinf\left(C_n\cart C_n\right)\le 2n$.
\end{theorem}
\begin{proof}
    Let $n\geq 18$ and let $G=C_n\cart C_n$.  
    
     For the upper bound, cops begin by occupying every vertex on the top and bottom rows of $G$.  On each subsequent cop turn, those cops who began on the top row move down, and those cops who began on the bottom row move up.  The robber cannot pass through the two rows of cops and hence will eventually be captured.
     
     For the lower bound, we give a strategy for the robber to evade $2n-25$ cops on $G$.  Let $n=6d+r$, for some positive integer $d$ and some $r$ in $\{0, \dots, 5\}$.  
    
    We partition the rows of $G$ into six {\it bands}, each consisting of either $d$ or $d+1$ consecutive rows.
    We call a row or column of $G$ \textit{nearly empty} if it contains at most one cop. Let $v_0$ denote the robber's initial position, and for $i\geq 1$, let $v_i$ denote the robber's position after his $i$th move.  The robber will choose each $v_i$ such that the following properties hold:
    \begin{enumerate}
        \item If the robber can reach a nearly empty row of $G$, then he can reach $v_i$.
        \item No cop is adjacent to $v_i$.
        \item After the cops' $(i+1)$st move, the robber can reach a nearly empty row of $G$ from $v_i$.\looseness=-1
    \end{enumerate}
    As in the proof of Theorem \ref{thm:2D_lower}, property (2) ensures that the robber cannot lose on the cops' $(i+1)$st turn, while properties (1) and (3) together ensure that, on his $(i+1)$st turn, he can reach $v_{i+1}$ from $v_i$.  Thus if the robber can choose the $v_i$ to have these three properties, then he can evade the cops indefinitely.
    
    Fix $i\geq 0$ and consider the state of the game after the cops' $i$th turn.  We explain how the robber chooses $v_i$.  Since there are at most $2n-25$ cops, there must exist a band of $d$ rows with at most $2d-5$ cops or a band of $d+1$ rows with at most $2(d+1)-5$ cops.  In either case, this band has at least three nearly empty rows.  Since  
    $$2(d+1)-5=2d-3<\frac{6d+r}{3}=\frac{n}{3},$$ 
    the band must have three consecutive empty columns. The robber chooses $v_i$ to be the vertex at the intersection of the second of these three columns and the second nearly empty row in the band. Clearly, $v_i$ satisfies property (2); we will next show that it also satisfies properties (1) and (3).
    
    To show that $v_i$ satisfies property (1), suppose the robber can reach a nearly empty row of $G$.  The robber would like to move to $v_i$, which is located on a nearly empty row in some band of $G$.  Since there are fewer than $2n-4$ cops playing on $G$, there must be at least three nearly empty columns in $G$.  At most one of these three columns may contain a cop in the robber's current row, and at most one of these columns may contain a cop in the row containing $v_i$.  Thus, using the nearly empty row, the robber can access at least one nearly empty column with no cop in either the robber's current row nor in the row containing $v_i$; using that column, he can reach the nearly empty row that contains $v_i$ and, consequently, $v_i$ itself, so (1) is satisfied.
    
    Finally, to show that $v_i$ satisfies property (3), suppose the cops have made their $(i+1)$st move.  Before the cops' move, $v_i$ was in the second of three consecutive empty columns in some band.  Thus, after the cops' move, $v_i$ is still in a column of the band that is empty except perhaps for the top and bottom rows in the band (since some cops may have moved into these rows on the cops' move).  Let $k$ denote the number of rows in the band.  There were at most $2k-5$ cops in the band before the cops' move, so disregarding the top and bottom rows of the band, there are still at most $2k-5$ cops spread across $k-2$ rows.  At least one of these rows must be nearly empty; the robber can reach this row using his current column in the band, so property (3) is satisfied.
\end{proof}
\end{section}

\begin{section}{Higher-Dimensional Grids}\label{sec:3D_grids}

We now turn our attention to three-dimensional Cartesian grids.  We focus on the $n \times n \times n$ grid $P_n \cart P_n \cart P_n$, but Theorem \ref{thm:monotone} can be used to establish similar results on three-dimensional grids of varying dimensions. 


Let $G=P_n\cart P_n\cart P_n$. We view the vertex set of $G$ as the set $\{(x,y,z) : 0\le x\le n-1, \, 0\le y\le n-1, \, 0\le z\le n-1\}$.  A \textit{row} of $G$ is a path $R_{y,z}$ consisting of those vertices whose second and third coordinates are $y$ and $z$, respectively.  Similarly, a \textit{column} of $G$ is a path $C_{x,z}$ consisting of all vertices whose first and third coordinates are $x$ and $z$, and a \textit{shaft} of $G$ is a path $S_{x,y}$ consisting of all vertices whose first and second coordinates are $x$ and $y$, respectively.  The \textit{$k$th plane} of $G$ is the subgraph induced by those vertices whose third coordinate is $k$.
 
 Our first main result in this section is an upper bound on $c_\infty(P_n\cart P_n\cart P_n)$.  In what follows, the following intuition will be helpful for the reader.  Often, an effective cop strategy works (very roughly) as follows.  First, the cops occupy a cut-set of the graph, thereby splitting the graph into two or more components.  They then identify the component in which the robber is located and gradually close in, making the robber's component smaller and smaller, until he is eventually captured.  To execute such a strategy with as few cops as possible, the cops will often want to find a ``small'' cut-set $S$ in the graph $G$ that minimizes the size of the largest component of $G\setminus S$.  In particular, the cops may wish to find a small cut-set $S$ that splits $G\setminus S$ into two components, each containing roughly half the vertices of $G$.

\begin{theorem}\label{thm:3D_grid_upper}
    For all positive integers $n$, we have $c_\infty(P_n\cart P_n\cart P_n)\leq (0.75+o(1))n^2$.
\end{theorem}
\begin{proof}
    Since we seek an asymptotic bound, by Theorem \ref{thm:monotone} it suffices to consider the case where $n$ is odd.  Let $G=P_n\cart P_n\cart P_n$; we give a strategy for $(0.75+o(1))n^2$ cops to catch a robber on $G$.  Loosely, the cops' strategy works as follows.  The cops will be divided into two groups: a group of $(3n^2+1)/4$ \textit{blocking cops} and a group of $(n+1)/2$ \textit{reserve cops}.  The blocking cops will occupy a cut-set $S$ of $G$, which we refer to as the cops' \textit{blockade}.  The robber must choose an initial position in some component of $G\setminus S$, which we refer to as the \textit{robber's component}.  The blockade prevents the robber from leaving his component.  Over the course of several rounds, the cops will gradually move the blockade so as to repeatedly decrease the size of the robber's component, leading to the robber's eventual capture.  
    
    For $m \in \{0, 1, \dots, 3(n-1)\}$, let $S_m$ be the set of all vertices in $G$ whose coordinates sum to $m$.  Let $m^* = 3(n-1)/2$; the cops' initial blockade will be the set $S_{m^*}$.  Initially, one blocking cop begins on each vertex of $S_{m^*}$.  To count the number of cops needed to occupy all of $S_{m^*}$, we must count the number of nonnegative-integer triples $(x,y,z)$ such that $x+y+z=m^*$.  This is equivalent to the number of $m^*$-combinations with repetition from 3 elements, which is $\binom{m^*+2}{2}$.  However, not every such triple corresponds to a vertex in $G$, since $x$, $y$, or $z$ might exceed $n-1$.  The number of these triples in which $x \ge n$ is equivalent to the number of $(m^*-n)$-combinations with repetition from 3 elements, which is $\binom{m^*-n+2}{2}$; by symmetry, this also counts the number of triples in which $y \ge n$ and those in which $z \ge n$.  Consequently, we have 
    \begin{align*}
    \size{S} &= \binom{m^*+2}{2}-3\binom{m^*-n+2}{2}\\
             &= \frac{(m^*+2)(m^*+1)}{2}-3\frac{(m^*-n+2)(m^*-n+1)}{2}\\
             &= \frac{(3n+1)(3n-1)}{8}-3\frac{(n+1)(n-1)}{8}\\
             &= \frac{9n^2-1}{8}-\frac{3n^2-3}{8}\\
             &= \frac{3n^2+1}{4},
    \end{align*}
    so there are enough blocking cops to occupy all vertices of $S_{m^*}$.  The reserve cops choose their initial positions arbitrarily.


     $S_{m^*}$ separates $G\setminus S_{m^*}$ into two components,  namely the subgraph of $G$ induced by those vertices whose coordinates sum to less than $m^*$ and the subgraph induced by those vertices whose coordinates sum to more than $m^*$.  (To see this, note that adjacent vertices in $G$ differ in only one coordinate, and in that coordinate, they differ by exactly one.  Consequently, the coordinate-sums of adjacent vertices differ by exactly 1, so no vertex with a coordinate-sum of $m^*-1$ or less can be adjacent to a vertex with coordinate-sum of $m^*+1$ or more.)  Suppose without loss of generality that the robber begins in the former component.

    Henceforth, the game will progress in \textit{phases}, each of which will consist of a finite number of rounds.  At the end of each phase, the blocking cops will move to decrease the size of the robber's component, while the reserve cops will help reinforce the blockade.  
    In the first phase, one reserve cop moves to each of the $(n+1)/2$ vertices of $S_{m*-1}$ whose third coordinate is $n-1$.  This may take several rounds; however, note that while the reserve cops are moving into position, the cops' blockade prevents the robber from leaving his component.  Once the reserve cops are in position, each reserve cop and each blocking cop on the $0$th plane remain in place, while all other blocking cops move down (i.e. they move so as to decrease their third coordinate).  
    
    Note that cops now occupy all vertices of $S_{m^*-1}$; we reclassify those cops in $S_{m^*-1}$ as blocking cops and the remaining cops as reserve cops.  In the next phase, the reserve cops move to cover all vertices of $S_{m^*-2}$ whose third coordinate is $n-1$, after which point the cops can shift the blockade downward to $S_{m^*-2}$.  The cops iterate this strategy, repeatedly shifting the blockade downward.  The robber's component gets smaller with each phase, so eventually he must be caught. 
    \end{proof}

Theorem \ref{thm:3D_grid_upper} established an upper bound on $\cinf(P_n \cart P_n \cart P_n)$; we now seek a lower bound.  As with $P_n \cart P_n$, the lower bound is considerably more complicated than the upper bound, so we start with a few preliminary lemmas.  

The cops' strategy in Theorem \ref{thm:3D_grid_upper} revolved around splitting $P_n \cart P_n \cart P_n$ into two equally-sized pieces as efficiently as possible, then slowly shrinking the robber's component.  To evade capture, it is crucial that the robber avoid this sort of situation.  As such, in the proof of Theorem \ref{thm:3D_grid_lower}, our lower bound on $\cinf(P_n \cart P_n \cart P_n)$, it will be useful to know how efficiently a fixed number of cops can split the graph into two pieces; even if the cops lack the numbers needed to split the graph into equally-sized pieces, we would like to know how close they can come.

Thus motivated, we introduce the following notation.
\begin{definition}
Fix positive integers $k$ and $c$ and integers $n_1, n_2, \dots, n_k$ all at least $2$, and let $G = P_{n_1} \cart P_{n_2} \cart \dots \cart P_{n_k}$.  Define $L(c; n_1, \dots, n_k)$ to be the minimum, over all subsets $S$ of $V(G)$ with $\size{S} \le c$, of the maximum size of a component in $G\setminus S$. 
\end{definition}

It is clear from the definition that $L(c; n_1, \dots, n_k)$ is nonincreasing in $c$.  Suppose $c$ is sufficiently small that $L(c; n_1, \dots, n_k) > \size{V(G)}/2$, and suppose we play on $G$ with $c$ cops.  In this situation, for every possible set $S$ of cop positions, some component of $G\setminus S$ contains at least half the vertices of $G$; we refer to this component as the {\it large component of $G$}.  (Note that the large component has size at least $L(c; n_1, \dots, n_k)$.)  

Furthermore, let $H$ denote the union of all components of $G\setminus S$ other than the large component.  We sometimes refer to $H$ as the {\it small component} of $G$.  (This is an abuse of notation because $H$ need not be connected; however, it will turn out that when the cops play ``efficiently'', the small component is in fact connected.)  In the proof of Theorem \ref{thm:3D_grid_lower}, we will need to be able to bound the sizes of the small and large components of various subgraphs of $P_n \cart P_n \cart P_n$, while playing against various numbers of cops.  To facilitate this, we introduce the following notation.

\begin{definition}
For a positive integer $m$, define
\begin{itemize}
    \item $C_m(n_1, \dots, n_k) = \size{\{(x_1,\dots,x_k) : x_1+\dots+x_k = m \text{ and } 0 \le x_i < n_i \text{ for all }i\}}$
    \item $S_m(n_1, \dots, n_k) = \size{\{(x_1,\dots,x_k) : x_1+\dots+x_k < m \text{ and } 0 \le x_i < n_i \text{ for all }i\}}$
    \item $L_m(n_1, \dots, n_k) = \size{\{(x_1,\dots,x_k) : x_1+\dots+x_k > m \text{ and } 0 \le x_i < n_i \text{ for all }i\}}$
\end{itemize}
\end{definition}

To motivate the choice of notation, let us remark that (as we will see shortly), when playing on $P_{n_1} \cart P_{n_2} \cart \dots \cart P_{n_k}$ with $C_m(n_1, \dots, n_k)$ cops, the size of the small component will be at most $S_m(n_1, \dots, n_k)$ and the size of the large component will be at least $L_m(n_1, \dots, n_k)$.

Bollob\'{a}s and Leader in \cite{BL91} completely characterized the sets $S$ of fixed size that maximize the size of the small component (and hence minimize the size of the large component).  We will not need the full power of their result; the following special case will suffice.

\begin{lemma}\label{lem:isoperimetric}
If $c \le C_m(n_1, \dots, n_k)$, then $L(c; n_1, \dots, n_k) \ge L_m(n_1, \dots, n_k)$.
\end{lemma}
\begin{proof}
    By Theorem 8 in \cite{BL91}, when $c = C_m(n_1, \dots, n_k)$, the maximum size of the small component is $S_m(n_1, \dots, n_k)$, realized when $S = \{(x_1,\dots,x_k) \in V(G)  : x_1+\dots+x_k = m\}$.  The large component consists of all vertices in $G$ that do not belong to $S$ and are not occupied by cops; consequently,
    $$L(c; n_1, \dots, n_k) = \prod_{i=1}^k n_k - S_m(n_1, \dots, n_k) - C_m(n_1, \dots, n_k) = L_m(n_1, \dots, n_k).$$
    When $c < C_m(n_1, \dots, n_k)$, the claim holds because $L(c; n_1, \dots, n_k)$ is nonincreasing in $c$.\looseness=-1
\end{proof}
Intuitively, Lemma \ref{lem:isoperimetric} states that the most ``efficient'' way for a limited number of cops to divide $G$ into two components is to occupy the set of vertices whose coordinates sum to $m$, for some $m$.

In the lemma below, we determine some specific values of $C_m$ and $S_m$ that will be useful later.
\begin{lemma}\label{lem:special_boxes}
Fix a positive integer $n$ and let $m = c(n-1)$ for some positive real number $c$.  
\begin{enumerate}
\item[(a)] If $c \le 1,$ then:
\begin{itemize}
\item $C_m(n,n) = \displaystyle (c+o(1))n$, and
\item $S_m(n,n) = \displaystyle \left (\frac{1}{2}c^2 + o(1)\right) n^2$.
\end{itemize}

\medskip\medskip
\item[(b)] If $n$ is even and $0.5 \le c \le 1,$ then: 
\begin{itemize}
\item $C_m(n/2-1,n/2-1,n) = \displaystyle \left (-\frac{1}{2}c^2+c-\frac{1}{4}+o(1)\right )n^2$, and
\item $S_m(n/2-1,n/2-1,n) = \displaystyle \left (-\frac{1}{6}c^3+\frac{1}{2}c^2-\frac{1}{4}c+\frac{1}{24}+o(1)\right )n^3$.
\end{itemize}

\medskip\medskip
\item[(c)] If $n$ is even and $0.5 \le c \le 1,$ then: 
\begin{itemize}
\item $C_m(n/2-1,n,n) = \displaystyle \left (\frac{1}{2}c-\frac{1}{8}+o(1)\right )n^2$, and
\item $S_m(n/2-1,n,n) = \displaystyle \left (\frac{1}{4}c^2-\frac{1}{8}c+\frac{1}{48}+o(1)\right )n^3$.
\end{itemize}

\medskip\medskip
\item[(d)] If $c \le 1$, then:
\begin{itemize}
\item $C_m(n,n,n) = \displaystyle \left (\frac{1}{2}c^2+o(1)\right )n^2$, and
\item $S_m(n,n,n) = \displaystyle \left (\frac{1}{6}c^3+o(1)\right )n^3$.
\end{itemize}
\medskip
If instead $1 < c \le 1.5,$ then: 
\begin{itemize}
\item $C_m(n,n,n) = \displaystyle \left (-c^2+3c-\frac{3}{2}+o(1)\right )n^2$, and
\item $S_m(n,n,n) = \displaystyle \left (-\frac{1}{3}c^3+\frac{3}{2}c^2-\frac{3}{2}c+\frac{1}{2}+o(1)\right )n^3$.
\end{itemize}
\end{enumerate}
\end{lemma}
\begin{proof}
\,

\begin{enumerate}
\item [(a)] $C_m(n,n)$ represents the number of ordered pairs $(x,y)$ with $0 \le x\le n-1, \, 0\le y \le n-1$ such that $x+y=m$, or equivalently, the number of ordered pairs $(x,m-x)$ such that $x\in\{0,\dots,m\}$; consequently, 
$$C_m(n,n)=m+1 = (c+o(1))n.$$ 
Similarly, $S_m(n,n)$ represents the number of ordered pairs $(x,y)$ with $0 \le x\le n-1, \, 0\le y \le n-1$ such that $x+y\le m$.  This is equivalent to the number of $m$-combinations with repetition from 3 elements, which is $\binom{m+2}{2}$.  Consequently, 
$$S_m(n,n)= \binom{m+2}{2} = \left (\frac{1}{2}c^2 + o(1)\right) n^2.$$

\item [(b)]  $C_m(n/2-1,n/2-1,n)$ represents the number of ordered triples $(x,y,z)$ with $0 \le x\le n/2-2, \, 0\le y \le n/2-2$, and $0 \le z \le n-1$ such that $x+y+z=m$.  The number of nonnegative integer ordered triples $(x,y,z)$ satisfying $x+y+z=m$ is equivalent to the number of $m$-combinations with repetition from 3 elements, which is $\binom{m+2}{2}$.  All such triples satisfy $x,y,z \ge 0$, as well as $z \le n-1$ (since $m \le n-1$); however, they do not all satisfy $0\le x\le n/2-2, \, 0\le y \le n/2-2$.  The number of these triples that {\it do not} satisfy $x \le n/2-2$ is equivalent to the number of $(m-n/2+1)$-combinations with repetition from 3 elements, which is $\binom{m-n/2+3}{2}$; by symmetry, this is also the number of triples not satisfying $y\le n/2-2$. Thus, we have
$$C_m(n/2-1,n/2-1,n) = \binom{m+2}{2} - 2\binom{m-n/2+3}{2} = \left (-\frac{1}{2}c^2+c-\frac{1}{4}+o(1)\right )n^2.$$

We can compute $S_m(n/2-1,n/2-1,n)$ similarly.  This time, we seek the number of solutions to $x+y+z \le m$ satisfying $0 \le x\le n/2-2, \, 0\le y \le n/2-2$ and $0 \le z \le n-1$.  The number of nonnegative integer solutions to $x+y+z \le m$ is equivalent to the number of $m$-combinations from 4 elements with repetition, which is $\binom{m+3}{3}$; the number not satisfying $x \le n/2-2$ (and, by symmetry, the number not satisfying $y \le n/2-2$) is $\binom{m-n/2+4}{3}$.  Thus,
$$S_m(n/2-1,n/2-1,n) = \binom{m+3}{3} - 2\binom{m-n/2+4}{3} = \left (-\frac{1}{6}c^3+\frac{1}{2}c^2-\frac{1}{4}c+\frac{1}{24}+o(1)\right )n^3.$$

\item [(c)] We can apply an argument similar to that in part (b); the only difference is that when computing both $C_m$ and $S_m$, we seek triples $(x,y,z)$ satisfying $0 \le x \le n/2-2$ and $0 \le y\le n-1, \, 0\le z \le n-1$, rather than $0 \le x\le n/2-2, \, 0\le y \le n/2-2$ and $0 \le z \le n-1$.  As in part (b), we thus have
$$C_m(n/2-1,n,n) = \binom{m+2}{2} - \binom{m-n/2+3}{2} = \left (\frac{1}{2}c-\frac{1}{8}+o(1)\right )n^2$$
and
$$S_m(n/2-1,n,n) = \binom{m+3}{3}-\binom{m-n/2+4}{3} = \left (\frac{1}{4}c^2-\frac{1}{8}c+\frac{1}{48}+o(1)\right )n^3.$$

\item [(d)] We can proceed as in parts (b) and (c), except that now we must satisfy the constraints $0 \le x\le n-1, \, 0\le y\le n-1, \, 0\le z \le n-1$.  As before, the number of solutions to $x+y+z=m$ is simply $\binom{m+2}{2}$ and the number of solutions to $x+y+z\le m$ is $\binom{m+3}{3}$.  If $c \le 1$, then the constraints do not exclude any solutions, so $C_m(n,n,n) = \binom{m+2}{2}$ and $S_m(n,n,n) = \binom{m+3}{3}$; if instead $1 < c \le 1.5$, then we must exclude solutions in which $x$, $y$, or $z$ exceeds $n-1$, hence $C_m(n,n,n) = \binom{m+2}{2}-3\binom{m-n+2}{2}$ and $S_m(n,n,n) = \binom{m+3}{3}-3\binom{m-n+3}{3}$.
\end{enumerate}
\end{proof}

We are nearly ready to establish a lower bound on $\cinf(P_n \cart P_n \cart P_n)$, but first we present some definitions that will be used in the proof.  Let $G = P_n \cart P_n \cart P_n$.  For a subgraph $H$ of $G$, let $\partial(H)$ denote the set of vertices in $H$ that have neighbors outside of $H$, and let $\text{int}(H)$ denote the subgraph induced by $V(H)\setminus \partial(H)$.

For a family $\mathcal{H}=\{H_1,H_2,\dots,H_k\}$ of subgraphs of $G$, we call $H_j$ the \textit{sparsest} member of $\mathcal{H}$ if $H_j$ contains the fewest cops among members of $\mathcal{H}$.  (A family may contain more than one sparsest member.)

Let $T$ denote the set of vertices of $G$ of the form $(x,y,z)$ with $z \ge \ceil{(n-1)/2}$.  The {\it top half} of $G$ is the subgraph of $G$ induced by $T$, while the {\it bottom half} is the subgraph induced by $V(G)\setminus T$.  Similarly, let $F$ be the set of vertices $(x,y,z)$ with $y \ge \ceil{(n-1)/2}$ and let $R$ be the set of vertices $(x,y,z)$ with $x \ge \ceil{(n-1)/2}$; the {\it front half} of $G$ is the subgraph induced by $F$, the {\it back half} is the subgraph induced by $V(G)\setminus F$, the {\it right half} is the subgraph induced by $R$, and the {\it left half} is the subgraph induced by $V(G)\setminus R$.  The {\it top-right quadrant} is the intersection of the top half and the right half; the {\it top-left quadrant}, {\it front-right quadrant}, etc. are defined analogously.  Finally, the {\it top-right-front octant} is the intersection of the top, right, and front halves of $G$; other octants are defined similarly.

We are finally ready to prove our lower bound on $\cinf(P_n \cart P_n \cart P_n)$.

\begin{theorem}\label{thm:3D_grid_lower}
For sufficiently large $n$, we have $c_\infty\left(P_n\cart P_n\cart P_n\right) >  0.7172n^2$.
\end{theorem}
\begin{proof}
        We suppose for convenience that $n$ is divisible by 10; the claim for $n$ not divisible by 10 can be handled with a similar argument in conjunction with Theorem \ref{thm:monotone}.  Let $G=P_n \cart P_n \cart P_n$.  We will give a strategy for the robber to evade $\left \lfloor 0.7172 n^2\right\rfloor$ cops on $G$.
        
        We begin with a high-level overview of the robber's strategy.  Let $v_0$ denote the robber's initial position, and for $i\geq 1$, let $v_i$ denote the robber's position after his $i$th move.  The robber will choose each $v_i$ so as to satisfy the following properties:
        \begin{enumerate}
            \item
                $v_i$ is in the large component of $G$ before the cops' $(i+1)$st move.
            \item
                No cop is adjacent to $v_i$.
            \item
                No matter how the cops move in their $(i+1)$st turn, $v_i$ will remain in the large component.
        \end{enumerate}
        The robber will play as follows.  To begin the game, he will look at the cops' initial positions and choose his own initial position $v_0$ satisfying properties (1)-(3).  Since $v_0$ satisfies property (2), the cops cannot capture him on their first turn.  After the cops move, the robber will choose $v_1$; he aims to move to $v_1$, so we must show that he can actually reach $v_1$ from $v_0$.  Since $v_0$ satisfied property (3) before the cops' turn, it must still belong to the large component of $G$.  Since $v_1$ satisfies property (1), it also belongs to the large component of $G$.  Thus, since both $v_0$ and $v_1$ belong to the large component, there must be a cop-free path from $v_0$ to $v_1$, hence the robber can reach $v_1$.  Once again, since $v_1$ satisfies property (2), the cops cannot capture the robber on their ensuing turn.  The robber waits until the cops move, chooses $v_2$, moves there, and so forth.  He can repeat this process, and thus evade capture, indefinitely.

    Fix $i\geq 0$; we explain how to choose $v_i$.  Consider the state of the game immediately after the cops' $i$th turn (or, if $i=0$, immediately after their initial placement).   First, the robber identifies the sparsest half of $G$, which we denote $H_i$.  Next, he identifies the sparsest quadrant of $G$ contained within $H_i$; denote this quadrant by $Q_i$.  Finally, he identifies the sparsest octant of $G$ contained within $Q_i$; denote this octant by $O_i$.  
    
    Since $O_i$ contains at most one-eighth of the total number of cops, it contains at most $0.08965n^2$ cops.  The $n/2$ planes in $O_i$ can be partitioned into $n/10$ groups of five consecutive planes.  One of these groups must contain at most $0.8965n$ cops; let $X_i$ denote the subgraph induced by some such group of planes.
    
    The robber seeks a subgrid of $O_i$ containing very few cops. 
    Since
    $$\displaystyle 0.8965\cdot n\cdot\frac{5}{\frac{n}{2}}=8.965<9,$$
    before the cops' turn, $X_i$ must contain an $n/2 \times 5 \times 5$ subgrid (that is, a subgrid isomorphic to $P_{n/2} \cart P_{5} \cart P_5$) containing fewer than nine cops.  (Refer to Figure \ref{fig:cube_lower}.)  Let $\hat{X_i}$ denote this subgrid.  Provided that 
    $n$ is sufficiently large, within $\hat{X_i}$ there must be a $3 \times 5 \times 5$ subgrid $\widetilde{X_i}$ that contains no cops.  The subgraph $\text{int}(\widetilde{X_i})$ is (or contains) a $3 \times 3 \times 1$ grid; the robber chooses $v_i$ to be the vertex in the center of this grid. 

\begin{center}    
\begin{figure}[ht]
\includegraphics[width=4.4cm]{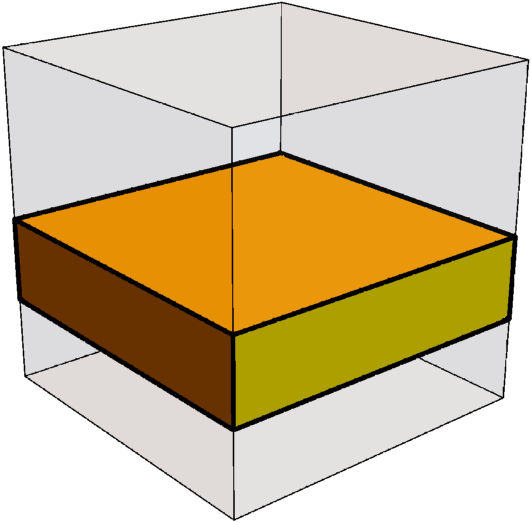}
\quad\quad\quad
\includegraphics[width=4.4cm]{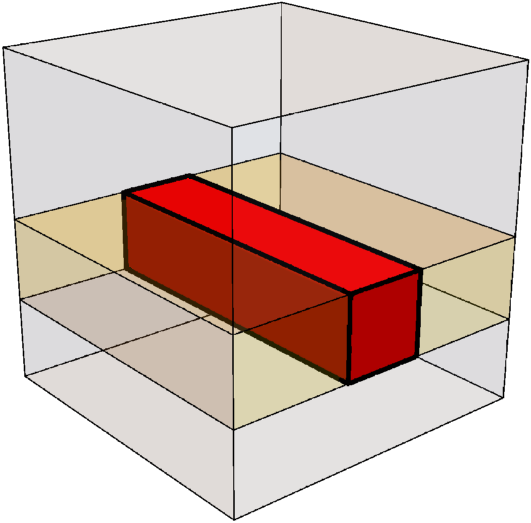}
\quad\quad\quad
\includegraphics[width=4.4cm]{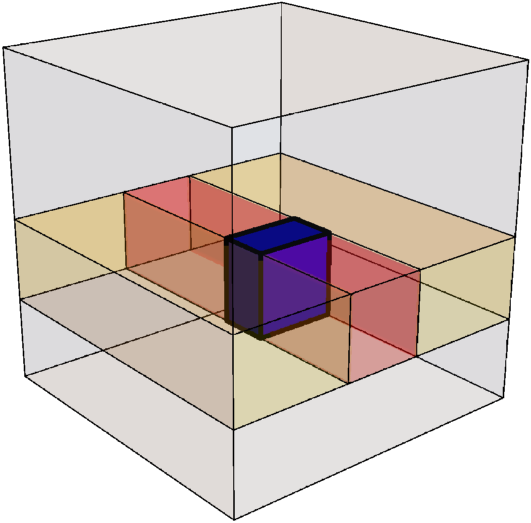}
\caption{Important subgraphs of $O_i$.  From left: $X_i$, $\hat{X_i}$, and $\widetilde{X_i}$.}
\label{fig:cube_lower}
\end{figure}
\end{center}
    
    Property (2) holds since all vertices adjacent to $v_i$ belong to $\widetilde{X_i}$, which has no cops.  We next verify properties (1) and (3).  Note that property (3) actually implies property (1): if we can show that $v_i$ must be in the large component of $G$ after the cops' $(i+1)$st move no matter how they move, then $v_i$ must have been in the large component of $G$ before the cops' $(i+1)$st move, since all of the cops could have chosen to stay put on their $(i+1)$st move.  We will now argue that (3), and therefore also (1), holds for $v_i$.
    
    The remainder of the proof contains some rather technical details.  To provide some intuition for the reader, we begin with a rough outline of what follows.  First, we argue that $v_i$ must belong to the large component of $\text{int}(O_i)$; we do so by showing that a robber at $v_i$ has access to more vertices of $\text{int}(O_i)$ than the small component could possibly contain.  We next argue that since the large component of $\text{int}(O_i)$ is larger than the small component of $\text{int}(Q_i)$, the robber has access to more vertices than are in the small component of $\text{int}(Q_i)$ and thus $v_i$ must be in the large component of $\text{int}(Q_i)$.  We use similar arguments to prove that $v_i$ must belong to the large component of $\text{int}(H_i)$ and, finally, to the large component of $G$.  
    
    We will now argue that $v_i$ must belong to the large component of $\text{int}(O_i)$ or, equivalently, that a robber at $v_i$ must be in the large component of $\text{int}(O_i)$.  There were no cops in $\widetilde{X_i}$ before the cops' move, so there can be no cops in $\text{int}(\widetilde{X_i})$ after the cops' move.  Hence, the robber can reach all vertices of $\text{int}(\widetilde{X_i})$ and can thus reach at least nine columns of $\text{int}(\hat{X}_i)$.  Likewise, there were at most eight cops in $\hat{X}_i$ before the cops' move, so there are at most eight cops in $\text{int}(\hat{X}_i)$ after the cops' move; thus one of the nine columns of $\text{int}(\hat{X}_i)$ that the robber can reach must be empty.  Moreover, the robber can reach all but at most eight vertices in some column of the second plane of $\text{int}(\hat{X}_i)$: either the empty column already belongs to the second plane, or it is directly above or below some column in the second plane, enabling the robber to reach any cop-free vertex in that column.
    
    
    
    We aim to show that the robber is in the large component of $\text{int}(O_i)$.  Toward this end, we first claim that he can reach an empty row or column within some plane $P$ of $\text{int}(O_i)$ that contains fewer than $0.3n$ cops.  There were at most $0.8965n$ cops in $X_i$ before the cops' turn, so there are at most $0.8965n$ cops in $\text{int}(X_i)$ after the cops' turn.  Since there are at least three planes in $\text{int}(X_i)$, one of the first three planes must have fewer than $0.3n$ cops; let $P$ denote some such plane.  We have already shown that the robber can reach an empty column in $\text{int}(X_i)$; let $P'$ denote the plane containing this column.  If $P' = P$, then we are done.  If instead $P$ and $P'$ are adjacent planes, then the robber can use the empty column in $P'$ to reach an empty row in $P$ and proceed from there to an empty column in $P$.  Finally, suppose $P$ and $P'$ are not adjacent.  It must be that $P$ is either the first or third plane of $\text{int}(\hat{X}_i)$.  We argued above that the robber can reach all but at most eight vertices of some column in the second plane of $\text{int}(\hat{X}_i)$; using this column, he can reach all but at most eight rows of $P$.  Since $P$ has $0.5n-1$ columns and fewer than $0.3n$ cops, it must have at least $0.2n$ empty rows.  If $n$ is sufficiently large, then $P$ has more than eight empty rows, so the robber can reach at least one, after which he may use this empty row to reach an empty column of $P$.  In any case, the robber can reach an empty column of $P$, as claimed.
    
    To prove that the robber must be in the large component of $\text{int}(O_i)$, we will show that the number of vertices of $\text{int}(O_i)$ to which he has access exceeds the number of vertices in the small component of $\text{int}(O_i)$.  As noted above, $P$ must have at least $0.2n$ empty rows and $0.2n$ empty columns.  The empty rows collectively contain at least $0.1n^2-0.2n$ vertices, the empty columns also contain at least $0.1n^2-0.2n$, and their intersection contains at most $0.04n^2$.  Consequently, the empty rows and columns collectively contain at least $0.16n^2-0.4n$ vertices, and the robber can access all of these.  Since $\text{int}(O_i)$ contains at most $0.08965n^2$ cops, no more than $0.08965n^2$ shafts can be nonempty, so the robber has access (through $P$) to at least $0.07035n^2-0.4n$ empty shafts in $\text{int}(O_i)$.  Each of these shafts contains $n/2-1$ vertices, so together they contain at least $(0.07035n^2 -0.4n)\cdot (0.5n-1)$ vertices.
    
    Note that $\text{int}(O_i)$ is a $(n/2-1) \times (n/2-1) \times (n/2-1)$ grid.  Hence, applying Lemma \ref{lem:special_boxes}(d) with $m=n/2-2$, we have \begin{align*}
        C_{m}(n/2-1,n/2-1,n/2-1)&=\left(\frac{1}{2}\cdot\left(1\right)^2+o(1)\right)\cdot \left(\frac{n}{2}-1\right)^2
        =\left(\frac{1}{8}+o(1)\right)n^2
    \end{align*}
    and
    \begin{align*}
    S_m(n/2-1,n/2-1,n/2-1) &=\left(\frac{1}{6}\cdot \left(1\right)^3+o(1)\right)\cdot\left(\frac{n}{2}-1\right)^3 = \left(\frac{1}{48}+o(1)\right)n^3,
    \end{align*}
    which is less than $0.021n^3$ provided that $n$ is sufficiently large. The number of cops in $\text{int}(O_i)$ is at most $0.08965 n^2$, which is smaller than $C_m(n/2-1,n/2-1,n/2-1)$.  Thus, by Lemma \ref{lem:isoperimetric}, the small component of $\text{int}(O_i)$ has fewer than $0.021n^3$ vertices.  The robber has access to more than $0.021n^3$ vertices, so he cannot be in the small component of $\text{int}(O_i)$; he must therefore be in the large component of $\text{int}(O_i)$.  For future reference, note also that the large component of $\text{int}(O_i)$ must contain more than $0.104n^3$ vertices (provided $n$ is sufficiently large), since $\text{int}(O_i)$ itself contains $(n/2-1)^3$ vertices, of which fewer than $0.021n^3$ belong to the small component and at most $0.08965 n^2$ are occupied by cops.  Consequently, the robber can access more than $0.104n^3$ vertices in $\text{int}(O_i)$.


    We next argue that the robber must be in the large component of $\text{int}(Q_i)$.  The number of cops in $\text{int}(Q_i)$ is at most $0.1793 n^2$, which is less than $0.2n^2$.  Note that $\text{int}(Q_i)$ is an $n \times (n/2-1) \times (n/2-1)$ grid.  Applying Lemma \ref{lem:special_boxes}(b) using $m=\frac{3}{4}(n-1)$, we have 
    \begin{align*}
        C_m(n/2-1,n/2-1,n)&=\left(-\frac{1}{2}\cdot\left (\frac{3}{4}\right )^2+\frac{3}{4}-\frac{1}{4}+o(1)\right)n^2
        =\left (\frac{7}{32}+o(1)\right)n^2,
    \end{align*}
    which is greater than $0.2n^2$ (provided $n$ is sufficiently large).
    In addition, 
    \begin{align*}
        S_m(n/2-1,n/2-1,n)&=\left(-\frac{1}{6}\cdot\left(\frac{3}{4}\right)^3+\frac{1}{2}\cdot\left(\frac{3}{4}\right)^2-\frac{1}{4}\cdot\frac{3}{4}+\frac{1}{24}+o(1)\right)n^3 = \left (\frac{25}{384}+o(1)\right )n^3,
    \end{align*}
    which is less than $0.1n^3$ (once again, provided $n$ is sufficiently large).  As noted earlier, the robber can reach more than $0.1n^3$ vertices in $\text{int}(O_i)$, and hence more than $0.1n^3$ vertices in $\text{int}(Q_i)$; thus, by Lemma \ref{lem:isoperimetric}, he must be in the large component of $\text{int}(Q_i)$.  Note that $\text{int}(Q_i)$ contains $(\frac{1}{4}+o(1))n^3$ vertices, of which fewer than $0.1n^3$ belong to the small component and only a quadratic number are occupied by cops; hence the large component of $\text{int}(Q_i)$ contains more than $0.15n^3$ vertices (provided $n$ is sufficiently large).
    
    A similar argument shows that the robber must be in the large component of $\text{int}(H_i)$.  The number of cops in $\text{int}(H_i)$ is at most $0.3586n^2$.  By Lemma \ref{lem:special_boxes}(c), using $m=0.96721(n-1)$, we have 
        \begin{align*}
            C_m(n/2-1,n,n)&=\left(\frac{1}{2}\cdot 0.96721-\frac{1}{8}+o(1)\right) n^2
                        > 0.3586n^2
        \end{align*}
        and 
        \begin{align*}
            S_m(n/2-1,n,n) &= \left(\frac{1}{4} \cdot 0.96721^2 - \frac{1}{8}\cdot 0.96721 + \frac{1}{48} + o(1)\right)n^3
               < 0.133807n^3
        \end{align*}
        (provided, as usual, that $n$ is sufficiently large).  The robber can reach more than $0.133807n^3$ vertices of $\text{int}(Q_i)$, so he must be in the large component of $\text{int}(H_i)$.  Of the $(\frac{1}{2}+o(1))n^3$ vertices in $\text{int}(H_i)$, under $0.133807n^3$ belong to the small component and at most a quadratic number are occupied by cops, so the large component contains at least $0.366193n^3$ vertices (assuming once again that $n$ is sufficiently large).
        
        Finally, we show that the robber must be in the large component of $G$, which will complete the proof.  By Lemma \ref{lem:special_boxes}(d) with $m=1.3189(n-1)$, we have
        \begin{align*}
            C_m(n,n,n)&=\left(-1.3189^2 + 3\cdot 1.3189 - \frac{3}{2}+o(1)\right) n^2
            > 0.7172n^2
        \end{align*}
        and 
        \begin{align*}
            S_m(n,n,n) &= \left(-\frac{1}{3} \cdot 1.3189^3 + \frac{3}{2}\cdot 1.3189^2 - \frac{3}{2}\cdot 1.3189 + \frac{1}{2} + o(1)\right)n^3
               < 0.36616n^3
        \end{align*}
        (as always, provided that $n$ is sufficiently large).  The robber can reach at least $0.366193n^3$ vertices in $\text{int}(H_i)$, and thus in $G$; this is greater than the number of vertices in the small component of $G$.  Hence the robber (and, by extension, $v_i$) must be in the large component of $G$.  We have thus shown that $v_i$ satisfies properties (1) and (3), which completes the proof.
\end{proof}

As a corollary of Theorem \ref{thm:3D_grid_lower}, we obtain a lower bound on the treewidth of $\cinf(P_n \cart P_n \cart P_n)$.  Alon and Mehrabian \cite{AM15} showed that for every graph $G$, we have $\cinf(G) \le \mathrm{tw}(G)+1$; consequently, our lower bound on $\cinf(P_n \cart P_n \cart P_n)$ also establishes a lower bound on $\mathrm{tw}(P_n \cart P_n)$:

\begin{corollary}\label{cor:treewidth}
For sufficiently large $n$, we have $\mathrm{tw}\left(P_n\cart P_n\cart P_n\right) \ge 0.7172n^2$.
\end{corollary}

To the best of our knowledge, the bound in Corollary \ref{cor:treewidth} is novel.
The treewidth of high-dimensional grid graphs was previously studied by Hickingbotham and Wood \cite{HW21}, who showed that for positive integers $n_1 \ge n_2 \ge \dots \ge n_d$,
$$\mathrm{tw}(P_{n_1} \cart P_{n_2} \cart \dots \cart P_{n_d}) = \Theta\left(\prod_{j=2}^d n_j\right);$$
however, although this result determines the order of growth of $\mathrm{tw}(P_{n_1} \cart P_{n_2} \cart \dots \cart P_{n_d})$, no attempt was made to pin down the coefficient on the leading term.  We also note that Otachi and Suda \cite{OS11} determined the pathwidth of three dimensional grid graphs, which yields the upper bound
$$\mathrm{tw}(P_n \cart P_n \cart P_n) \le \mathrm{pw}(P_n \cart P_n \cart P_n) = \floor{\frac{3n^2+2n}{4}}.$$

\subsection{Four or more dimensions}

We conclude the section with a brief look at higher-dimensional grids.  The ideas used to prove Theorem~\ref{thm:3D_grid_upper} extend naturally to higher dimensions, as we next demonstrate.

\begin{theorem}\label{thm:4D_grid_upper}
If $G$ is the $d$-fold Cartesian product of copies of $P_n$, then
$$c_{\infty}(G) \le (1+o(1)) \sum_{k=0}^{\floor{d/2}}(-1)^k \binom{d}{k}\binom{\floor{d(n+1)/2} - kn - 1}{d-1}.$$
\end{theorem}
\begin{proof}
As in the proof of Theorem \ref{thm:3D_grid_upper}, we suppose for the sake of simplicity that $n$ is odd; a similar argument suffices when $n$ is even. For nonnegative integers $a$ and $b$, let $S_{a,b}$ denote the set of vertices in the $a$-fold Cartesian product of copies of $P_n$ whose coordinates sum to $b$.  As in Theorem \ref{thm:3D_grid_upper}, $S_{a,b}$ separates $G\setminus S_{a,b}$ into two separate components: one consisting of those vertices whose coordinates sum to less than $b$ and one consisting of those vertices whose coordinates sum to greater than $b$.  We claim that 
$$\size{S_{a,b}} = \sum_{k=0}^{\floor{a/2}}(-1)^k \binom{a}{k}\binom{b - kn + a - 1}{a-1}.$$
$\size{S_{a,b}}$ counts the number of solutions to $x_1 + x_2 + \dots + x_a = b$ with each $x_i$  in $\{0, \dots, n-1\}$.  The number of nonnegative-integer solutions to $x_1 + x_2 + \dots + x_a = b$ is $\binom{b+a-1}{a-1}$; however, we must exclude solutions for which some $x_i$ exceeds $n-1$.  For $k \in \{1, \dots, \floor{a/2}\}$, the number of solutions in which $k$ different $x_i$ exceed $n-1$ is equivalent to the number of nonnegative-integer solutions to $x_1 + x_2 + \dots + x_a = b-kn$, which is $\binom{b-kn+a-1}{a-1}$.  Consequently, by the inclusion-exclusion principle, $$\size{S_{a,b}} = \sum_{k=0}^{\floor{a/2}}(-1)^k \binom{a}{k}\binom{b - kn + a - 1}{a-1},$$
as claimed.  


We are now ready to present the cops' strategy.  Let $m = \floor{d(n-1)/2}$. We begin with $\size{S_{d,m}}+\size{S_{d-1,m-n}}$ cops, of which $\size{S_{d,m}}$ are designated \textit{blocking cops} and the remainder are designated \textit{reserve cops}.  Initially, one blocking cop occupies each vertex of $S_{d,m}$ and the reserve cops position themselves arbitrarily.  Let $v$ denote the robber's initial position.  Let $H$ denote the component of $G\setminus S_{d,m}$ consisting of all vertices whose coordinates sum to less than $m$, and suppose that the robber begins in $H$ (a similar argument works if he begins in the other component of $G\setminus S_{d,m}$). Over the course of several turns, reserve cops move to occupy every vertex in $H$ whose coordinates sum up to $m-1$ and whose first coordinate is $n-1$.  (Note that there are precisely $\size{S_{d-1,m-n}}$ such vertices, so there are enough reserve cops to occupy all of them.)  The robber must remain in $H$ during this time, since the blocking cops separate $H$ from $G\setminus H$.  

Once the reserve cops are in position, all blocking cops whose first coordinate is at least 1 move so as to decrease their first coordinate.  The cops now occupy all vertices of $S_{d,m-1}$.  All cops in $S_{d,m-1}$ are now designated blocking cops, while the rest are designated reserve cops.  The cops repeat this process: the reserve cops move to occupy vertices with first coordinate $n-1$ whose coordinates sum to $m-2$, the blocking cops decrease their first coordinates, and so on.  The size of the robber's component gradually decreases, until he is eventually captured.
\end{proof}

We remark that the techniques used to establish a lower bound on $c_{\infty}(P_n \cart P_n \cart P_n)$ should also extend to the $d$-dimensional grid, but a general analysis seems difficult; rather, the argument would likely need to be specialized to each particular value of $d$.

\end{section}

\begin{section}{Hypercubes}\label{sec:hypercubes}

We now turn our attention to the infinite-speed cop number of hypercubes.  The $n$-dimensional {\it hypercube} $Q_n$ is the graph whose vertex set consists of all $n$-bit binary strings, with two vertices adjacent if and only if the corresponding strings differ in exactly one bit.  (Equivalently, it is the $n$-fold Cartesian product of copies of $K_2$.)  Mehrabian \cite{Meh12} established the following bounds on $c_{\infty}(Q_n)$:
\begin{theorem}[\cite{Meh12}, Theorem 9.1(e)]\label{thm:mehrabian_hypercube}
For all $n \ge 1,$ there exist positive constants $\eta_1$ and $\eta_2$ such that  
$$\eta_1 \cdot \frac{2^n}{n^{3/2}} \le c_{\infty}(Q_n) \le \eta_2 \cdot \frac{2^n}{n}.$$
\end{theorem}

The upper bound in Theorem~\ref{thm:mehrabian_hypercube} is straightforward: the infinite-speed cop number of any graph is bounded above by the domination number -- since the cops can start on a dominating set and capture the robber with their first move -- and it is well-known that $\gamma(Q_n) \le \frac{2^{n+1}}{n+1}$.  The lower bound, however, is more intricate and relies on a connection between the infinite-speed cop number of a graph and its treewidth.  We obtain a stronger lower bound using a different approach.  Inspired by the proof of Theorem 2.1 in \cite{BB15}, we show how the robber can use a potential function to evade a large number of cops.

Before proceeding, we state a technical lemma.  We will need to make use of Harper's theorem on vertex-isoperimetry in $Q_n$ \cite{Har66}, which loosely states that out of all subsets of $V(Q_n)$ having fixed size, Hamming balls are the sets with smallest closed neighborhood.  Consequently,
 for $S \subseteq V(Q_n)$, if $\size{S} \ge \sum_{i=0}^{k-1} \binom{n}{i}$ for some positive integer $k$, then $\size{N[S]} \ge \sum_{i=0}^{k} \binom{n}{i}$.

\begin{lemma}\label{lem:cube_small_component}
Let $n$ be even, and suppose we play the infinite-speed robber game on $Q_n$ with $c$ cops.  For sufficiently large $n$, if $\displaystyle c \le \frac{2^{n-3}}{n \ln n}$, then the small component has fewer than $\displaystyle \frac{1}{2}\binom{n}{n/2}$ vertices.
\end{lemma}
\begin{proof}
It is a well-known consequence of Stirling's approximation that $\binom{n}{n/2} \sim \frac{2^n}{\sqrt{n\pi/2}}$; we will make use of this fact below.

Let $k = \frac{n}{2}-\frac{1}{2}\sqrt{n (\ln n + \ln \ln n)}$.  We claim that $c \le \binom{n}{k}$.  Stirling's approximation yields (see \cite{Spe14}, p. 66, equation (5.41)): 
\begin{align*}
\binom{n}{k} &\sim \binom{n}{n/2}\exp\left (-(\sqrt{n(\ln n + \ln \ln n)})^2/(2n)\right )\\
             &= \binom{n}{n/2}\exp\left (-\frac{1}{2} (\ln n + \ln \ln n)\right )\\
             &= \binom{n}{n/2} \cdot \frac{1}{\sqrt{n\ln n}}\\
             &\sim \frac{2^n}{\sqrt{n\pi/2}}\cdot\frac{1}{\sqrt{n\ln n}}\\
             &=\frac{1}{\sqrt{\pi/2}}\cdot \frac{2^n}{n\sqrt{\ln n}}.
\end{align*}
Thus $c \le \binom{n}{k}$ provided that $n$ is sufficiently large.  Since the cops occupy the boundary of the small component -- that is, the set of vertices in the small component with neighbors in the large component -- we have that the boundary of the small component has size at most $\binom{n}{k}$.  Consequently, by Harper's theorem, the size of the small component cannot exceed the size of a Hamming ball with boundary of size $\binom{n}{k}$, which is $\sum_{i=0}^k \binom{n}{i}$.   We bound this sum using the Chernoff bound, in a form originally stated by Okamoto (\cite{Oka59}, Theorem 1).  Let $X$ be a binomial random variable with parameters $n$ and $p=1/2$.  Now 
\begin{align*}
\sum_{i=0}^k \binom{n}{i} &= 2^n \cdot P(X \le k)\\
                          &= 2^n \cdot P\left (X \le \left (\frac{1}{2}-\frac{1}{2}\sqrt{\frac{\ln n + \ln \ln n}{n}}\right )n\right )\\ 
                          &\le 2^n\exp\left (-2\left(\frac{1}{2}\sqrt{\frac{\ln n+\ln\ln n}{n}}\right)^2 n\right)\\ 
                          &= 2^n\exp\left (-\frac{1}{2}(\ln n+\ln\ln n)\right )\\
                          &=\frac{2^n}{\sqrt{n\ln n}}.
\end{align*}
In particular, the size of the small component is
less than $\frac{1}{2}\binom{n}{n/2}$ for sufficiently large $n$. 
\end{proof}

We are now ready to establish a lower bound on $\cinf(Q_n)$.

\begin{theorem}\label{thm:hypercubes}
For sufficiently large $n$, we have $\displaystyle c_{\infty}(Q_n) \ge \frac{2^{n-3}}{n \ln n}$.
\end{theorem}
\begin{proof}
We give a strategy for the robber to avoid $\frac{2^{n-3}}{n \ln n}-1$ or fewer cops.  Intuitively, on each robber turn, the robber wants to move to a vertex with few enough cops nearby that the cops cannot cut off the robber's escape routes.  To formalize the robber's strategy, we use a potential function.  Suppose the robber currently occupies vertex $v$.  For a cop $c$ at vertex $u$, define that cop's \textit{potential} $P(c)$ as follows.  If in fact $u=v$, meaning that the cop and robber occupy the same vertex, then $P(c) = 1$; otherwise, $P(c) = \binom{n}{\textrm{dist}(u,v)-1}^{-1}$.  Moreover, we define the \textit{total potential}, $\Phi$, by $\Phi = \sum_c P(c),$ where the sum is taken over all cops $c$. 

Note that any cop on or adjacent to the robber's vertex contributes 1 to the total potential.  Hence, if $\Phi < 1$ at the end of a robber turn, then there are no cops in the closed neighborhood of the robber's vertex, so he cannot be captured on the next cop turn.  In general, the smaller the value of $\Phi$, the fewer cops there are ``near'' the robber; as such, the robber will aim to keep $\Phi$ as small as possible.   Before outlining the robber's strategy, we prove the following statement: if $\Phi < 1/2$ at the end of a robber turn, then the robber's vertex belongs to the large component of $Q_n$, and it will continue to belong to the large component even after the next cop turn.  

Suppose $\Phi < 1/2$ at the end of a robber turn and suppose (by symmetry) that the robber currently occupies the vertex whose coordinates are all zero.  We claim that this vertex is currently in the large component and remains in the large component after the cops take their turn.  By Lemma \ref{lem:cube_small_component}, the size of the small component is strictly less than $\frac{1}{2}\binom{n}{n/2}$.  Hence it suffices to show that, after the cops' move, the robber can reach at least this many vertices.  

Let $S$ denote the set of all vertices at distance $n/2$ from the robber and let $\mathcal{P}$ denote the collection of all shortest paths from the robber's vertex to a vertex in $S$.  As the robber traverses such a path, each step corresponds to changing some coordinate of his position from 0 to 1.  Consequently, each path in $\mathcal{P}$ corresponds to changing $n/2$ coordinates from 0 to 1, in some order; there are $\binom{n}{n/2}$ ways to decide which coordinates to change and $(n/2)!$ possible orders in which to change them, so 
$$\size{\mathcal{P}} = \binom{n}{n/2} \cdot (n/2)! = \frac{n!}{(n/2)!(n/2)!}\cdot (n/2)! = \frac{n!}{(n/2)!}.$$

Not all of these paths can actually be traversed by the robber, because they may contain one or more vertices occupied by cops.  Say that a path in $\mathcal{P}$ is {\it blocked} by a cop if that cop occupies some vertex on the path.  We claim that if $\Phi < 1/2$ at the end of the robber's turn, then fewer than half of the paths will be blocked after the cops' ensuing turn.  To see this, first note that a cop at distance $k$ from the robber's vertex (where $k \le n/2$) blocks exactly $k! \cdot \binom{n-k}{n/2-k} \cdot (n/2-k)!$ paths in $\mathcal{P}$: there are $k!$ shortest paths from the robber's vertex to the cop's, and there are $\binom{n-k}{n/2-k} \cdot (n/2-k)!$ shortest paths from the cop's vertex to a vertex in $S$.  This quantity simplifies:
$$k! \cdot \binom{n-k}{n/2-k} \cdot (n/2-k)! = k! \cdot \frac{(n-k)!}{(n/2-k)!(n/2)!} \cdot (n/2-k)!
                                           = \frac{k!(n-k)!}{(n/2)!}.$$
The total number of paths blocked by the cops is at most the sum of this quantity over all cops.  For a cop $c$, let $d(c)$ denote that cop's distance from the robber's vertex.  We now have
\begin{align*}
\textrm{Total paths blocked } &\le \sum_c \frac{(d(c))!(n-d(c))!}{(n/2)!}\\
                              &=   \sum_c \frac{n!}{n!} \cdot \frac{(d(c))!(n-d(c))!}{(n/2)!}\\
                              &=   \frac{n!}{(n/2)!} \sum_c \cdot \frac{1}{\binom{n}{d(c)}}\\
                              &\le \frac{n!}{(n/2)!} \sum_c P(c)\\
                              &=   \frac{n!}{(n/2)!} \cdot \Phi\\
                              &< \frac{n!}{(n/2)!} \cdot \frac{1}{2},
\end{align*}
where the inequality on the fourth line follows because a cop that is currently at distance $d(c)$ was at distance $d(c)+1$ or less prior to the cops' turn and thus contributed at least $\binom{n}{d(c)+1-1}^{-1}$ to $\Phi$.

Consequently, at least half of the shortest paths from the robber's vertex to $S$ are unblocked, so the robber can reach at least half of the vertices in $S$.  Since $\size{S} = \binom{n}{n/2}$, the robber cannot be in the small component (which, as mentioned above, has size less than $\frac{1}{2}\binom{n}{n/2}$).

We next claim that on each robber turn, if the robber is within the large component of $Q_n$, then he can move so as to make the total potential less than $1/2$.  Let $\Phi_v$ denote the value of the total potential when the robber moves to $v$; we aim to show that $\Phi_v < 1/2$ for some vertex $v$ in the large component.  We begin by computing $\sum_{v \in V(G)} \Phi_v$.  For each cop $c$ and for all $k \in \{0, 1, \dots, n\}$, there are $\binom{n}{k}$ vertices at distance $k$ from $c$; $c$ exerts a potential of $\binom{n}{k-1}^{-1}$ on each such vertex provided $k \ge 1$ and a potential of 1 otherwise.  Consequently, between all of the vertices in the graph, the cop exerts an aggregate potential of 
\begin{align*}
1 + \sum_{k=1}^{n} \frac{\binom{n}{k}}{\binom{n}{k-1}}
      &= 1 + \sum_{k=1}^n \frac{n-k+1}{k}\\
      &= 1 + \sum_{k=1}^n \left (\frac{n+1}{k} - 1\right )\\
      &= 1 - n + (n+1)\sum_{k=1}^n \frac{1}{k}\\ 
      &\le 1 - n + (n+1)(\ln n + 1)\\
      &\le 2n\ln n \quad \text{(provided $n \ge 3$)}.
\end{align*}
Since there are $\frac{2^{n-3}}{n \ln n}$ cops, the aggregate potential exerted by all cops is at most
$$\frac{2^{n-3}}{n\ln n} \cdot 2n\ln n = 2^{n-2}.$$
Consequently, there are at most $2^{n-1}$ vertices $u$ such that $\Phi_u \ge 1/2$, so there are at least $2^{n-1}$ vertices $v$ such that $\Phi_v < 1/2$.  Since the large component contains more than $2^{n-1}$ vertices, we have $\Phi_{v^*} < 1/2$ for at least one vertex $v^*$ in the large component.  Since the robber is currently in the large component, there is a cop-free path from his current vertex to $v^*$; consequently, the robber may move to $v^*$.

Finally, we explain the robber's actual strategy.  Initially, the robber begins at any vertex $v$ such that $\Phi_v < 1/2$.  Note that the cops cannot capture the robber on their ensuing turn, since any cop within distance 1 of the robber would contribute 1 to the potential, contradicting $\Phi_v < 1/2$.  Moreover, since the potential is less than 1/2, the robber will be in the large component after the cops' move.  On the robber's next turn, he moves to a vertex $v^*$ such that $\Phi_{v^*} < 1/2$.  he then repeats this process indefinitely, repeatedly moving so as to keep the potential below $1/2$. 
\end{proof}

\end{section}

\begin{section}{Future Work}\label{sec:conclusion}
   
    We conclude by proposing several open problems and directions for future study of the infinite-speed game.
    
    \begin{itemize}
    \item Theorem \ref{thm:2D_grids} determines $\cinf(P_n \cart P_n)$ exactly when $n$ is odd or 2.  However, when $n$ is even and at least 4, we know only that $\cinf(P_n \cart P_n) \in \{n-1,n\}$.  What is the exact value of $\cinf(P_n \cart P_n)$ is in this case?  (We suspect that the answer should be $n$ for all even values of $n$.)\\
    
    \item Similarly, Theorems \ref{thm:3D_grid_upper} and \ref{thm:3D_grid_lower} show that $0.7172n^2\le \cinf(P_n\cart P_n\cart P_n)\le (0.75+o(1))n^2$.  What are the exact asymptotics of $\cinf(P_n \cart P_n \cart P_n)$?  (We suspect that in fact $\cinf(P_n\cart P_n\cart P_n) = (0.75+o(1))n^2$.)\\
    
    Tightening the lower bound in Theorem \ref{thm:3D_grid_lower} would also improve Corollary \ref{cor:treewidth}; if one could prove that in fact $\cinf(P_n\cart P_n\cart P_n) = (0.75+o(1))n^2$, then it would follow that $\mathrm{tw}(P_n\cart P_n\cart P_n) = (0.75+o(1))n^2$.\\
            
    \item Theorem \ref{thm:4D_grid_upper} gives a general upper bound for $\cinf(G)$ when $G$ is the $d$-fold Cartesian product of $P_n$.  While we do not have a closed form for the summation in this upper bound, empirical data suggests that the bound is proportional to $n^{d-1}/\sqrt{d}$.  It would be interesting to understand how the asymptotics of $\cinf(G)$ change with $d$.  In particular, is there some constant $c$ such that $\cinf(G)=(c+o(1))n^{d-1}/\sqrt{d}$?\\
    
    \item All of our results involve Cartesian products of paths and cycles.  What can be said about the behavior of $\cinf$ under other graph products, such as the strong or categorical products?\\
    
    \end{itemize}
\end{section}

\section*{Acknowledgements}
The authors are grateful to the anonymous referees for their helpful comments, which have significantly improved the paper.

\end{document}